\documentclass[11pt]{amsart}

\usepackage{amsfonts,amssymb,amsthm,eucal,color,bbm,verbatim,graphicx,hyperref}
\usepackage[french,english]{babel}

\usepackage[margin=1.25in]{geometry}
\calclayout

\newtheorem{thm}{Theorem}

\newtheorem{lemma}[thm]{Lemma}

\theoremstyle{definition}

\newcommand{\R}{\mathbb{R}}
\newcommand{\E}{\mathbb{E}}
\newcommand{\energy}{\mathcal{E}}

\newcommand{\N}{\mathbb{N}}

\newcommand{\C}{\mathbb{C}}
\renewcommand{\P}{\mathbb{P}}

\DeclareMathOperator{\Hess}{Hess}
\newcommand{\inprod}[2]{\left\langle #1, #2 \right\rangle}

\newcommand{\Unitary}[1]{\mathbb{U}\left(#1\right)}

\newcommand{\ind}[1]{\mathbbm{1}_{#1}}

\allowdisplaybreaks[2]
\thanks{\footnotemark {$^\dagger$} Supported in part by NSF DMS 1612589.}
\author{Elizabeth Meckes{$^\dagger$}}

\author{Kathryn Stewart{$^\dagger$}}

\address{Department of Mathematics, Applied Mathematics, and
  Statistics, Case Western Reserve University, 10900 Euclid Ave.,
  Cleveland, Ohio 44106, U.S.A.}

\email{elizabeth.meckes@case.edu}

\address{Department of Mathematics, Applied Mathematics, and
  Statistics, Case Western Reserve University, 10900 Euclid Ave.,
  Cleveland, Ohio 44106, U.S.A.}

\email{kathrynstewart@case.edu}

\title{On the eigenvalues of truncations of random unitary matrices}

\begin{document}

\maketitle

\begin{abstract}
We consider the empirical eigenvalue distribution of an $m\times m$
principle submatrix of an $n\times n$ random unitary matrix
distributed according to Haar measure.  Earlier work of Petz and R\'effy
identified the limiting spectral measure if $\frac{m}{n}\to\alpha$, as
$n\to\infty$; under suitable scaling, the family $\{\mu_\alpha\}_{\alpha\in(0,1)}$ of limiting
measures interpolates between uniform measure on the unit disc (for small
$\alpha$) and uniform measure on the unit circle (as $\alpha\to1$).
In this note, we prove an explicit concentration inequality which
shows that for fixed $n$ and $m$, the bounded-Lipschitz distance
between the empirical spectral measure and the corresponding
$\mu_\alpha$ is typically of order $\sqrt{\frac{\log(m)}{m}}$ or
smaller.  The approach is via the theory of two-dimensional Coulomb
gases and makes use of a new ``Coulomb transport inequality'' due to
Chafa\"i, Hardy, and Ma\"ida.  

\end{abstract}

\section{Introduction}
Let $U$ be an $n\times n$ Haar-distributed unitary matrix.
By a truncation of such a matrix, we mean a reduction to the
upper-left $m\times m$ block, for some $m\le n$.  In the case that
$m=o(\sqrt{n})$, the truncated matrix is close to a matrix of
independent, identically distributed Gaussian random variables (see Jiang \cite{Jia09}); the circular law for the Ginibre ensemble would lead one to expect that the eigenvalue distribution was approximately uniform in a disc, and  this was indeed verified by Jiang in \cite{Jia09}.  At the opposite extreme, namely $m=n$, we have the full original matrix $U$.  The eigenvalues of $U$ itself are also well-understood; it was first proved by
Diaconis and Shahshahani \cite{DiaSha94} that for a sequence $\{U_n\}$
with $U_n$ distributed according to Haar measure on $\Unitary{n}$, the
corresponding sequence of  empirical spectral
measures converges to the uniform measure on the circle, weakly in
probability. In more recent work \cite{MM13} of the first author and M.\ Meckes, it was shown that if $\mu_n$ denotes the spectral
measure of $U$  and $\nu$ is the
uniform measure on the circle, then with probability one, for $n$
large enough,
$W_1(\mu_n,\nu)\le\frac{C\sqrt{\log(n)}}{n}$ (here, $W_1(\cdot,\cdot)$
is the $L^1$-Wasserstein distance; the definition is given at the end of this
section). This result demonstrates  a stronger uniformity of the eigenvalues of such a matrix than, for example, a collection of $n$ i.i.d.\ uniform points on the circle (whose empirical measure typically has distance of the order $\frac{1}{\sqrt{n}}$ from the uniform measure).

It is thus natural to consider the evolution of the distribution of
the eigenvalues of an $m\times m$ trucation of $U$, as
$\alpha=\frac{m}{n}$ ranges from $o(1)$ to $1-o(1)$.  Figure \ref{F:ews}
shows simulations of the eigenvalues of truncations for various values
of $\alpha$: in it, one can see the thinning out of the distribution
in the center of the disc, as more of the original matrix is kept
and the eigenvalues move from being uniform on a disc to uniform on
the circle.

\begin{figure}\label{F:ews}
   \includegraphics[width=1.5in]{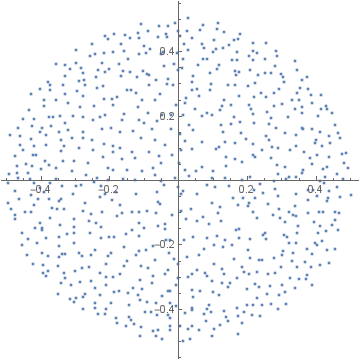}\qquad\includegraphics[width=1.5in]{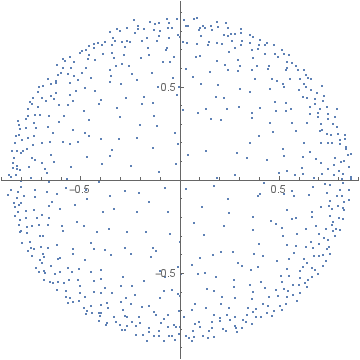}\qquad\includegraphics[width=1.5in]{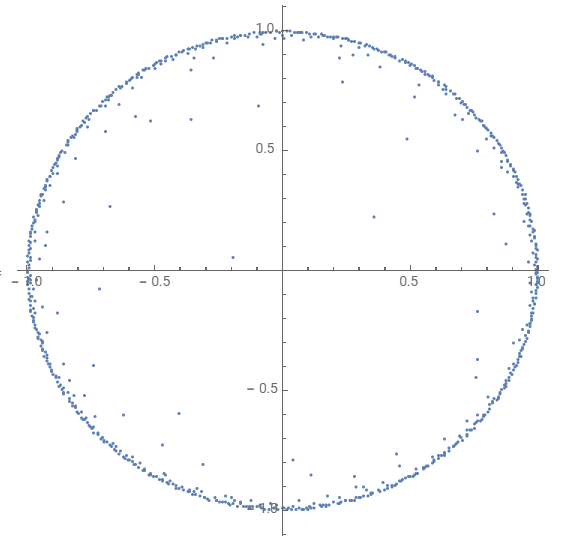}\caption{The
   eigenvalues of an $m\times m$  truncation of a $n\times n$
   Haar-distributed unitary matrix, with $\frac{m}{n}=.25$,
   $\frac{m}{n}=.75$, and$\frac{m}{n}\approx.99$.}
\end{figure}

In fact, the exact eigenvalue distribution of an $m\times m$ submatrix
with $m<n$ is
known (see \cite{ZycSom00}; also \cite{PetRef05}).  For our purposes,
it is most natural to consider the eigenvalues of an $m\times m$
truncation rescaled by $\sqrt{\frac{n}{m}}$; under this scaling, the
joint density of the eigenvalues is supported on $\left\{|z|<\sqrt{\frac{n}{m}}\right\}^n$ and has
density there given by 
\begin{equation}\label{E:density-unnormalized}\frac{1}{c_{n,m}}\prod_{1\le j<k\le m}|z_j-z_k|^2\prod_{j=1}^m\left(1-\frac{m}{n}|z_j|^2\right)^{n-m-1},\end{equation}
where 
\[c_{n,m}=\pi^m m!\left(\frac{n}{m}\right)^{ m(m+1)/2}\prod_{j=0}^{m-1}\binom{n-m+j-1}{j}^{-1}\frac{1}{n-m+j}.\]

Petz and R\'effy \cite{PetRef05} made use of the explicit eigenvalue
density to identify the large-$n$ limiting spectral measure, when
$\frac{m}{n}\to\alpha\in(0,1)$; it has radial density with respect to
Lebesgue measure on $\C$ (as it must, by rotation-invariance), given by 
\begin{equation}\label{E:falpha}f_\alpha(z)=\begin{cases}\frac{(1-\alpha)}{\pi(1-\alpha|z|^2)^2},& 0<|z|<1;\\0,& \mbox{otherwise}.\end{cases}\end{equation}

While the mathematical motivation in studying the eigenvalues of these truncations, and particularly the evolution of the ensemble as the ratio $\frac{m}{n}$ ranges from $o(1)$ to $1-o(1)$, is clear, there are also many phyical systems in which large unitary matrices play a
central role, and in which truncations of those matrices arise
naturally.  E.g., in chaotic scattering,  the amplitudes of waves coming into the
system are related to the amplitudes of outgoing waves by a large unitary matrix
(called an $S$-matrix), and the so-called transmission matrix (related
to long-lived resonances of the system) is a truncation of the
$S$-matrix.  See, e.g., \cite{FyoSom97}, where the use of random
unitary matrices in this context was explored.  

\medskip

The purpose of this paper is to give non-asymptotic results; i.e., to describe the ensemble of eigenvalues of truncations of $U\in\Unitary{n}$ for fixed (large) $n$.  Our main result on approximation of the spectral measure is the following.

\begin{thm} \label{T:main}
Let
$n,m\in\N$ with $1\le m<n$. Let
$U\in\Unitary{n}$ be distributed according to Haar measure, and let
$z_1, \ldots,
z_m$ denote the eigenvalues of the top-left $m\times m$ block of
$\sqrt{\frac{n}{m}}U$.  The joint law of $z_1,\ldots,z_m$ is denoted
$\P_{n,m}$. 

Let $\hat{\mu}_m$ be the empirical spectral measure given by 
\begin{equation*}
\hat{\mu}_m := \frac{1}{m} \sum_{i=1}^m \delta_{z_i}.
\end{equation*}
Let $\alpha=\frac{m}{n}$ and let $\mu_\alpha$ be the probability
measure on the unit disc with the density $f_\alpha$ defined in
\eqref{E:falpha}.   For any $r>0$,
\begin{equation*}
\mathbb{P}_{n,m} \Big[ d_{BL} \left( \hat{\mu}_m , \mu_{\alpha} \right) \geq r \Big]  \le e^2\exp\left\{-C_\alpha m^2r^2+2m\log(m)+C_\alpha'm\right\}+\frac{e}{2\pi}\sqrt{\frac{m}{1-\alpha}}e^{-m},\end{equation*}
where $C_\alpha=\frac{1}{128\pi(1+\sqrt{3+\log(\alpha^{-1})})^2}$ and $C_\alpha'=6+3\log(\alpha^{-1})$. \\
\end{thm}

The bounds in Theorem \ref{T:main} are tight enough that we can in fact treat the evolution of the process of spectral measures of truncations of $U$, as the truncation ratio $\alpha$ ranges from $o(1)$ to $1-o(1)$.

\begin{thm}\label{T:union}
Let $U$ be an $n\times n$ Haar-distributed matrix in $\Unitary{n}$
and, for  $1\le m < n$, let $\hat{\mu}_m$ be the empirical spectral measure
 of the top-left $m\times m$ block of$\sqrt{\frac{n}{m}}U$.  Let $\alpha=\frac{m}{n}$, 
 $\mu_{\alpha}$ be the probability measure on the disc with
 density $f_{\alpha}$ as in \eqref{E:falpha}, and let
 $\{k_n\}_{n\ge 1}\subseteq\N$ be such that $k_n=o(n)$ and $\frac{k_n}{\log(n)^2}\to\infty$.   Then with probability
1, for $n$ large enough,
\begin{align*}
d_{BL} \left( \hat{\mu}_m , \mu_{\alpha} \right) & \leq \delta_m
\end{align*}
for every $m \in\{k_n,\ldots,n-1\}$, where
\begin{align*}
\delta_m & = \begin{cases} 48\sqrt{ \frac{2\pi \log m}{m}}, & m \geq \frac{n}{e} ;\\
\frac{165\sqrt{\log(\frac{n}{m})\log(n)}}{\sqrt{m}}, & m < \frac{n}{e}.\end{cases}
\end{align*}
\end{thm}

Note that if $m=o(n)$, then $\delta_m$ is of the order
$\sqrt{\frac{\log(n)^2}{m}}$.  The
restriction on $k_n$ in the statement of the theorem thus implies that
with probability one,
\[\sup_{k_n\le m\le n-1}d_{BL}\left( \hat{\mu}_m , \mu_{\alpha} \right)\xrightarrow{n\to\infty}0.\]

Observe that, although the support of the empirical spectral measure
$\hat{\mu}_m$ is the disc of radius $\sqrt{\frac{n}{m}}$, the limiting spectral measure is supported on the unit disc; the following treats the question of how far into this intermediate regime the eigenvalues are likely to stray.

\begin{thm}\label{T:edge}
Let $z_1,\ldots,z_m$ be the eigenvalues of the top-left $m\times m$ block of $\sqrt{\frac{n}{m}}U$,
with joint law $\P_{n,m}$, and let $\alpha=\frac{m}{n}$.  Then
for any $\epsilon\in\left(0,\frac{1}{\sqrt{\alpha}}-1\right)$,
\[\P_{n,m}\left[\max_{1\le j\le m}|z_j|>1+\epsilon\right]\le\frac{e \left(1 - \alpha^m(1+\epsilon)^{2m} \right)}{2\pi\sqrt{n\alpha(1-\alpha)}\left(1 - \alpha(1+\epsilon)^2 \right)}\left[\frac{\left(1 - \alpha(1+\epsilon)^2 \right)^{(1-\alpha)}}{(1-\alpha)^{1-\alpha}\alpha^\alpha}\right]^n.\]
If $\epsilon\ge\frac{1}{\sqrt{\alpha}}-1$, then 
\[\P_{n,m}\left[\max_{1\le j\le m}|z_j|>1+\epsilon\right]=0.\]
\end{thm}

This estimate requires some effort to parse.  Firstly, observe that choosing $\epsilon$ so that 
\begin{equation}\label{E:epsilon-lower-bound}(1+\epsilon)^2>\frac{1}{\alpha}\left[1-(1-\alpha)\left(\frac{\alpha}{e}\right)^{\frac{\alpha}{1-\alpha}}\right]\end{equation}
gives that 
\begin{equation*}\P_{n,m}\left[\max_{1\le j\le
      m}|z_j|>1+\epsilon\right]\le\frac{e \left(1 -
      \alpha^m(1+\epsilon)^{2m}
    \right)}{2\pi\sqrt{n\alpha(1-\alpha)}\left(1 -
      \alpha(1+\epsilon)^2 \right)}e^{-\alpha n}.\end{equation*}

Note that in the non-trivial case that $\alpha(1+\epsilon)^2<1$,
\[\frac{\left(1 - \alpha^m(1+\epsilon)^{2m} \right)}{\left(1 - \alpha(1+\epsilon)^2 \right)}=1+\alpha(1+\epsilon)^2+\cdot+(\alpha(1+\epsilon)^2)^{m-1}\le m,\]
so that
\begin{equation}\label{E:tail-bound-2}\P_{n,m}\left[\max_{1\le j\le
      m}|z_j|>1+\epsilon\right]\le\frac{e}{2\pi}\sqrt{\frac{m}{(1-\alpha)}}e^{-m}.\end{equation}
While the bound stated in Theorem \ref{T:edge} is formally stronger, we will use the simpler bound \eqref{E:tail-bound-2} in the following discussion, separated into three distinct regimes.
\begin{enumerate}
\item $m=o(n)$:

For $\frac{m}{n}=\alpha$ small, the lower bound on $(1+\epsilon)^2$ in \eqref{E:epsilon-lower-bound} is
\begin{align*}\frac{1}{\alpha}&\left[1-(1-\alpha)\left(\frac{\alpha}{e}\right)^{\frac{\alpha}{1-\alpha}}\right]\\&=\frac{1}{\alpha}\left[1-(1-\alpha)\exp\left\{\frac{\alpha}{1-\alpha}(\log(\alpha)-1)\right\}\right]\\&=\frac{1}{\alpha}\left[1-(1-\alpha)\left(1+\frac{\alpha}{1-\alpha}(\log(\alpha)-1)+O\left(\left(\frac{\alpha}{1-\alpha}(\log(\alpha)-1)\right)^2\right)\right)\right]\\&=2-\log(\alpha)+o(1).\end{align*}
If $m=o(n)$ and $m\ge2\log(n)$, then the bound in
\eqref{E:tail-bound-2} tends to zero at least as quickly as $n^{-\frac{3}{2}}$, and so it follows from the Borel--Cantelli lemma that if $m_n$ is any sequence with $m_n=o(n)$ and $m_n\ge 2\log(n)$, then for any $\delta>0$, with probability one,  for $n$ large enough the support of the empirical spectral measure $\hat{\mu}_{m_n}$ lies within the disc of radius $\sqrt{2+\delta+\log(\frac {n}{m_n})}$, as opposed to the \emph{a priori} support of the disc of radius $\sqrt{\frac{n}{m_n}}$.

\item There are $c>0$ and $C<1$ such that $cn\le m\le Cn$:

Here the bound in \eqref{E:tail-bound-2} tends to zero exponentially
with $n$, and in this case the lower bound on $\epsilon$ from \eqref{E:epsilon-lower-bound} results in a fixed radius $r_m$ (somewhat smaller than
$\sqrt{\frac{n}{m}}$ but still bounded away from one, in terms of $c$ and $C$), such that, if $m_n$ is a sequence with $cn\le m_n\le Cn$ for all $n$, then with probability one for $n$ large enough, $\hat{\mu}_{m_n}$ is supported in a disc of radius $r_{m_n}$.

\item $\frac{m}{n}\to 1$:

The bound in \eqref{E:tail-bound-2} tends to zero exponentially with
$n$, and 
for $\alpha$ tending to one, 
\[\frac{1}{\alpha}\left[1-(1-\alpha)\left(\frac{\alpha}{e}\right)^{\frac{\alpha}{1-\alpha}}\right]=2-\alpha+O((1-\alpha)^2).\]
It thus follows from the Borel--Cantelli lemma that for any
$\epsilon>0$, if $m_n$ is a
sequence with $m_n<n$ for each $n$ and $\frac{m_n}{n}\to1$, then with probability one, for $n$ large enough, the empirical spectral measure $\hat{\mu}_{m_n}$ is contained within a disc of radius $1+\epsilon$.

\end{enumerate}

\noindent {\bf Definitions and notation.}  Throughout the paper, $U$
will denote a Haar-distributed random unitary matrix in $\Unitary{n}$
and, for $1\le m < n$, $z_1,\ldots,z_m$ will denote the eigenvalues
of the top-left $m\times m$ block of $\sqrt{\frac{n}{m}}U$, with
associated spectral measure $\hat{\mu}_m$.  

The $L^1$-Wasserstein distance between probability measures $\mu$ and
$\nu$ is given by 
\[W_1(\mu,\nu)=\sup_{f}\left|\int fd\mu-\int fd\nu\right|,\]
where the supremum is taken over Lipschitz functions with Lipschitz
constant 1.
The bounded-Lipschitz distance between $\mu$ and $\nu$ is given by 
\[d_{BL}(\mu,\nu)=\sup_{f}\left|\int fd\mu-\int fd\nu\right|,\]
where the supremum is taken over functions which are bounded by 1 and
have Lipschitz constant bounded by 1.

We will use the following uniform version of Stirling's approximation,
which is an easy consequence of equation (9.15) in \cite{Feller}.

\begin{lemma} \label{L: stirling} For each positive integer $n$,
\begin{equation*}
\sqrt{2\pi}n^{n+\frac{1}{2}}e^{-n} \leq n! \leq en^{n+\frac{1}{2}}e^{-n}.
\end{equation*}
\end{lemma}

\section{Proofs of the Main Results}

We begin with Theorem \ref{T:edge}.
\begin{proof}[Proof of Theorem \ref{T:edge}]
The form of the eigenvalue density
\eqref{E:density-unnormalized}, specifically the presence of the
Vandermonde determinant, gives that $z_1,\ldots,z_m$ form a
determinantal point process on $\C$ with the kernel (with respect to
Lebesgue measure)

\begin{equation*} \begin{split}
K_{n,m}(z_1, z_2) & = \frac{m}{n} \sum_{j=1}^m \frac{1}{N_j} \left(\frac{m}{n}z_1\overline{z_2} \right)^{j-1} \left( 1-\frac{m}{n}|z_1|^2 \right)^{\frac{n-m-1}{2}}\left( 1-\frac{m}{n}|z_2|^2 \right)^{\frac{n-m-1}{2}} \\
& \qquad \qquad \times \ind{(0,\infty)}\left( 1-\frac{m}{n}|z_1|^2 \right) \ind{(0,\infty)}\left( 1-\frac{m}{n}|z_2|^2 \right),
\end{split}\end{equation*}

where the normalization factor $N_j$ is given by

\begin{equation*}
N_j = \frac{\pi (j-1)!(n-m-1)!}{(n-m+j-1)!}.
\end{equation*}

Let $B_{r}$ denote the ball of radius $r$, and let $\epsilon\in\left(0,\frac{1}{\sqrt{\alpha}}-1\right)$. Then the expected number of
$z_i$ outside $B_{1+\epsilon}$ is given by

\begin{equation*} \begin{split}
\E [ \mathcal{N}_{B_{1+\epsilon}^c}] & = \int_{B_{1+\epsilon}^c} K_{n,m}(z,z) dz \\
&= 2\pi \int_{1+\epsilon}^{\frac{1}{\sqrt{\alpha}}} \sum_{j=1}^m \frac{1}{N_j} \alpha^j r^{2(j-1)} \left( 1 - \alpha r^2 \right)^{n-m-1} r dr \\
& \leq 2\pi \left( 1 - \alpha(1+\epsilon)^2 \right)^{n-m-1}\sum_{j=1}^m \frac{1}{N_j} \alpha^j\int_{1+\epsilon}^{\frac{1}{\sqrt{\alpha}}} r^{2(j-1)} r dr \\
& = 2\pi \left( 1 - \alpha(1+\epsilon)^2\right)^{n-m-1} \sum_{j=1}^m \frac{1}{N_j} \left( \frac{1}{2j} \left(1-\alpha^j(1+\epsilon)^{2j} \right) \right) \\&\le 2\pi \left(1 - \alpha(1+\epsilon)^2 \right)^{n-m-1}\left(1 - \alpha^m(1+\epsilon)^{2m} \right) \sum_{j=1}^m \frac{1}{(2j)N_j}.
\end{split}\end{equation*}

The sum on the right can be computed using the hockey stick identity:

\begin{equation*}\begin{split}
\sum_{j=1}^m \frac{1}{(2j)N_j} & = \sum_{j=1}^m \frac{(n-m+j-1)!}{2\pi j!(n-m-1)!} \\
& = \frac{1}{2\pi}\sum_{j=1}^m\binom{n-m+j-1}{n-m-1}\\
& =\frac{1}{2\pi}\sum_{k=n-m}^{n-1}\binom{k}{n-m-1}\\
& =\frac{1}{2\pi}\left[\binom{n}{n-m}-1\right].
\end{split}\end{equation*}

Then

\begin{equation*}\begin{split}
\E [ \mathcal{N}_{B_{1+\epsilon}^c}] & \leq \left(1 - \alpha(1+\epsilon)^2 \right)^{n-m-1}\left(1 - \alpha^m(1+\epsilon)^{2m} \right)\binom{n}{m}\\&=\frac{\left(1 - \alpha^m(1+\epsilon)^{2m} \right)}{\left(1 - \alpha(1+\epsilon)^2 \right)}\left(1 - \alpha(1+\epsilon)^2 \right)^{(1-\alpha)n}\binom{n}{m}.
\end{split}\end{equation*}
The version of Stirling's formula in Lemma \ref{L: stirling} gives that, for $m=\alpha n$,
\[\binom{n}{m}\le \frac{e}{2\pi\sqrt{n\alpha(1-\alpha)}}\left[\frac{1}{(1-\alpha)^{1-\alpha}\alpha^\alpha}\right]^n.\]
and so by Markov's inequality, 
\begin{align*} \begin{split}
\P\left[\max_{1\le j\le m}|z_j|\ge 1+\epsilon \right]& = \mathbb{P} \left[ \mathcal{N}_{B_{1 + \epsilon}^c}  \geq 1 \right] \\
& \le \tfrac{e \left(1 - \alpha^m(1+\epsilon)^{2m} \right)}{2\pi\sqrt{n\alpha(1-\alpha)}\left(1 - \alpha(1+\epsilon)^2 \right)}\left[\tfrac{\left(1 - \alpha(1+\epsilon)^2 \right)^{(1-\alpha)}}{(1-\alpha)^{1-\alpha}\alpha^\alpha}\right]^n.
\end{split} \end{align*}

If $\epsilon\ge\frac{1}{\sqrt{\alpha}}-1$, then 
\[\P\left[\max_{1\le j\le m}|z_j|\ge 1+\epsilon \right]\le \P\left[\max_{1\le j\le m}|z_j|\ge\frac{1}{\sqrt{\alpha}}\right]=0,\]
since the eigenvalues of a principal submatrix of $U$ necessarily have modulus bounded by 1.

\end{proof}

We now proceed with Theorem \ref{T:main}.  The proof is an adaptation of the approach in \cite{ChaHarMai16}, using the framework of Coulomb gases.  Specifically,  
the form of the eigenvalue density given in Equation \eqref{E:density-unnormalized}  means that the $z_i$ can be viewed as the (random) locations of $m$ unit charges in a two-dimensional Coulomb gas with external potential, as follows.  If the energy $H_{n,m}(z_1,\ldots,z_m)$ is defined by 
\[H_{n,m}(z_1,\ldots,z_m)=\sum_{j\neq k}\log\left(\frac{1}{|z_j-z_k|}\right)+m\sum_{j=1}^mV_{n,m}(z_j),\]
with the potential $V_{n,m}(z)$ defined by
\[V_{n,m}(z)=\begin{cases}-\frac{n-m-1}{m}\log\left(1-\frac{m}{n}|z|^2\right),& |z|<\sqrt{\frac{n}{m}};\\\infty,& \mbox{otherwise},\end{cases}\]
then the Gibbs measure on $\C^m$ (taking the inverse temperature
$\beta$ to be 2) is
\[d\P_{n,m}(z_1,\ldots,z_m)=\frac{1}{Z_{n,m}}\exp\left\{-H_{n,m}(z_1,\ldots,z_m)\right\}d\lambda(z_1)\ldots d\lambda(z_m),\]
where $\lambda$ denotes Lebesgue measure on $\C$.  That is, the Gibbs
measure in this Coulomb gas model is exactly the same as the density of the eigenvalues of the top-left $m\times m$ block of $\sqrt{\frac{n}{m}}U$, and so the empirical measure of the charges
$z_1,\ldots,z_m$ has the same distribution as the empirical spectral
measure $\hat{\mu}_m$.  
This was  the viewpoint taken by Petz and R\'effy in \cite{PetRef05} to identify the large-$n$ limiting spectral measure; the limiting measure with density $f_\alpha$ as in \eqref{E:falpha} is exactly the equilibrium measure for
the 2-dimensional Coulomb gas model with potential
\[V_\alpha=\begin{cases}-\left(\frac{1}{\alpha}-1\right)\log\left(1-\alpha|z|^2\right),& |z|<\frac{1}{\sqrt{\alpha}};\\\infty,& \mbox{otherwise},\end{cases}.\]
 It
should be noted that the viewpoint here is slightly
removed from the usual Coulomb gas model, where the potential would
not depend on $m$ or $n$; allowing such a dependence is possible
because the approach taken here is non-asymptotic; i.e., $n$ and
$m$ are fixed throughout.

In recent work, Chafa\"i, Hardy and Ma\"ida \cite{ChaHarMai16} have
developed an approach to studying the non-asymptotic behavior of
Coulomb gases, using new inequalities they call Coulomb transport
inequalities.  Specifically, if $\energy(\mu)=\iint g(x-y)d\mu(x)d\mu(y)$ is the Coulomb energy,
with 
\[g(x)=\begin{cases}\log\frac{1}{|x|},&
  d=2;\\\frac{1}{|x|^{d-2}},&d>2\end{cases}\]
the $d$-dimensional Coulomb kernel, they showed that if $D$ is a compact
subset of $\R^d$, then there is a constant $C_D>0$ such that for any
pair of probability measures $\mu$ and $\nu$ supported on $D$ with
$\mathcal{E}(\mu), \mathcal{E}(\nu)<\infty,$
\[W_1(\mu,\nu)^2\le C_D\energy(\mu-\nu).\]

When comparing to the equilibrium measure $\mu_V$ of the Coulomb gas
model with potential $V$, this leads to the estimate
\begin{equation}\label{E:V-energy-distance-est}
d_{BL}(\mu,\mu_V)^2\le W_1(\mu,\mu_V)^2\le C_V\left[\energy_V(\mu)-\energy_V(\mu_V)\right],
\end{equation}
where $\energy_V$ is the modified energy functional
\begin{equation}\label{E:modified-energy}\energy_V(\mu)=\energy(\mu)+\int Vd\mu.\end{equation}

The estimate \eqref{E:V-energy-distance-est} is the key ingredient in
the proof of Theorem \ref{T:main}.  The proof follows the analysis in
\cite{ChaHarMai16} closely, although their analysis does not apply
directly to our potential.  In particular, certain technical lemmas in
\cite{ChaHarMai16}, e.g., Theorem 1.9, require
modifications because boundedness assumptions made there
are not satisfied by
$V_{n,m}$.

The central idea of the proof of Theorem \ref{T:main} is the following simple application of
the bound \eqref{E:V-energy-distance-est}.  Let ${\bf z}=(z_1,\ldots,z_m)$, and
let $\hat{\mu}_{\bf z}:=\frac{1}{m}\sum_{j=1}^m\delta_{z_j}$.  Let $\mu_\alpha$ have density $f_\alpha$, for $\alpha=\frac{m}{n}$.  Given $r>0$, 
\begin{align*}\P&\left[d_{BL}(\hat{\mu}_{\bf z},\mu_\alpha)>r\right]\\&\qquad=\frac{1}{Z_{n,m}}\int_{d_{BL}(\hat{\mu}_{\bf z},\mu_\alpha)>r}\exp\left\{-\sum_{j\neq
  k}\log\left(\frac{1}{|z_j-z_k|}\right)-m\sum_{j=1}^mV_{n,m}(z_j)\right\}d\lambda^n(z)\\&\qquad\approx \frac{1}{Z_{n,m}}\int_{d_{BL}(\hat{\mu}_{\bf z},\mu_\alpha)>r}\exp\left\{-m^2\energy_{V_{n,m}}(\hat{\mu}_{\bf z})\right\}d\lambda^n(z),\end{align*}
and on $\left\{z:d_{BL}(\hat{\mu}_{\bf z},\mu_\alpha)>r\right\}$,
\eqref{E:V-energy-distance-est} gives that
\[\exp\left\{-m^2\energy_{V_{n,m}}(\hat{\mu}_{\bf z})\right\}\le\exp\left\{-cm^2r^2+m^2\energy_{V_{n,m}}(\mu_\alpha)\right\}.\]
Of course, since the measures $\hat{\mu}_{\bf z}$ are singular, the approximate
inequality above is invalid, and so part of the argument is to mollify
the empirical measures under consideration.  Since our potential
$V_{n,m}$ is only finite on $\left\{|z|<\sqrt{\frac{n}{m}}\right\}$,
this requires in particular that the probability of any eigenvalues lying too close
to the boundary of this disc is small, which follows from Theorem \ref{T:edge}.  In fact,some further truncation is useful in order to obtain improved control on the constants.  Beyond that, all that is
really needed is to give estimates for the normalizing constant and
the modified Coulomb energy at the equilibrium measure.

The following lemma relates the energy
$H_{n,m}(x_1,\ldots,x_m)$ to the modified Coulomb energy of
the mollified spectral measure.

\begin{lemma}\label{L:reg} For ${\bf z}=(z_1,\ldots,z_m)\in\C^m$, let
  $\hat{\mu}_{\bf z}=\frac{1}{m}\sum_{j=1}^m\delta_{z_j}$.  That is,
  $\hat{\mu}_{\bf z}$ is the probability measure putting equal mass at each of
  the $z_j$.
For any $\epsilon > 0$, define 
\begin{equation*}
\hat{\mu}_{\bf z}^{(\epsilon)} := \hat{\mu}_{\bf z} \ast \lambda_{\epsilon},
\end{equation*}
where $\lambda_{\epsilon}$ is the uniform probability measure on the
ball $B_{\epsilon}$. Then for $z_1,\ldots,z_m\in B_{\frac{1}{\sqrt{\alpha}}-\sqrt{\epsilon}}$, with $\epsilon<\left(\frac{\alpha}{4+2\sqrt{\alpha}}\right)^2$,
\begin{equation*}
H_{n,m} \left( z_1, \ldots, z_m \right) \geq m^2
\mathcal{E}_{V_{n,m}} \left( \hat{\mu}_{\bf z}^{(\epsilon)} \right) - m
\mathcal{E} \left( \lambda_{\epsilon} \right) - \frac{m^2(1-\alpha)\epsilon}{2\alpha}.
\end{equation*}
\end{lemma}

\begin{proof}Lemma 4.2 from
\cite{ChaHarMai16} gives that 
\begin{equation*}
H_{n,m} \left( z_1, \ldots, z_m \right) \geq m^2 \mathcal{E}_{V_{n,m}} \left( \hat{\mu}_{\bf z}^{(\epsilon)} \right) - m \mathcal{E} \left( \lambda_{\epsilon} \right) - m \sum_{i=1}^m \left( V_{n,m} \ast \lambda_{\epsilon} - V_{n,m} \right) \left( z_i \right),
\end{equation*}
so that the only task is to give an upper bound for $\left(V_{n,m} \ast \lambda_{\epsilon} - V_{n,m} \right) \left( z_i \right)$.

Let $0<\epsilon<\left(\frac{\alpha}{4+2\sqrt{\alpha}}\right)^2$, and suppose that $z<\frac{1}{\sqrt{\alpha}}-\sqrt{\epsilon}$.  Then in particular, $V_{n,m}(y)<\infty$ for $|y-z|<\epsilon$, so that 
\begin{equation*}\begin{split}
\left( V_{n,m} \ast \lambda_{\epsilon} - V_{n,m} \right) \left( z \right)
& = \int \left( V_{n,m} (z-\epsilon u) - V_{n,m}(z) \right) d\lambda_1 (u). 
\end{split}\end{equation*}
Note that by symmetry, $\int \inprod{\nabla V_{n,m} (z)}{u}
d\lambda_1(u) =0$ for fixed $z$. Moreover,  $V_{n,m}$ is convex,
so that $\Hess V_{n,m}$ is positive semi-definite; it thus follows from
Taylor's theorem that
\begin{equation*}\begin{split}
\left( V_{n,m} \ast \lambda_{\epsilon} - V_{n,m}\right) \left( z \right) & \leq \frac{\epsilon^2}{2} \underset{|y-z| \leq \epsilon}{\underset{y \in \R^2} {\sup}} \int \inprod{\Hess(V_{n,m})_yu}{u} d\lambda_1(u)\\
&\le \frac{\epsilon^2}{2} \underset{|y|<\frac{1}{\sqrt{\alpha}}-\sqrt{\epsilon}+\epsilon}{\sup} \frac{1}{4} \Delta V_{n,m}(y).
\end{split}\end{equation*}
If  $|y|<\frac{1}{\sqrt{\alpha}}-\sqrt{\epsilon}+\epsilon$, then
\[\Delta V_{n,m}(y)=\frac{4\left(1-\alpha-\frac{1}{n}\right)}{(1-\alpha|y|^2)^2}\le\frac{4(1-\alpha)}{(2\sqrt{\alpha}(\sqrt{\epsilon}-\epsilon)-\alpha(\sqrt{\epsilon}-\epsilon)^2)^2}\le\frac{4(1-\alpha)}{\alpha\epsilon},\]
for $\epsilon$ in the range specified above.

\end{proof}

In the proof of Theorem \ref{T:main}, we will use the following version of the Coulomb transport inequality
from \cite{ChaHarMai16}, which is an immediate consequence of Lemma 3.1 together with Theorem
1.1 of that paper.  The lemma refers to an \emph{admissible} external
potential $V$; we refer the reader to \cite{ChaHarMai16} for the
definition, which is satisfied for our potentials $V_{n,m}$.  A key
fact is that such a potential is associated with an equilibrium
measure $\mu_V$, which is the unique minimizer of the modified energy
$\mathcal{E}_V$ as defined in \eqref{E:modified-energy}.  For our potential
$V_{n,m}$, the equilibrium measure is $\mu_{\alpha}$ for
$\alpha=\frac{m}{n}$. 
\begin{lemma}[Coulomb Transport Inequality \cite{ChaHarMai16}] \label{L:transport}
Let $V$ be an admissible external potentialon $\R^d$ with equilibrium
measure $\mu_V$.  If $D \subset \mathbb{R}^d$ is compact then for any $\mu \in \mathcal{P}(\R^d)$ supported in $D$, 
\begin{equation*}
d_{BL}(\mu, \mu_{V})^2 \leq W_1(\mu, \mu_{V})^2 \leq C_{D\cup supp(\mu_{V})} \left( \mathcal{E}_{V}(\mu) - \mathcal{E}_{V}(\mu_{V}) \right),
\end{equation*}
where if $R>0$ is such that $D \cup supp(\mu_{V}) \subset B_R$,
then $C_{D\cup supp(\mu_{V})}$ can be taken to be $vol(B_{4R})$.
\end{lemma}

\begin{proof}[Proof of Theorem \ref{T:main}]
Fix $r>0$.  By Theorem \ref{T:edge} and the discussion following it, it is possible to choose
$\eta_\alpha>0$ such that $1+\eta_\alpha<\frac{1}{\sqrt{\alpha}}$ and
so that
\[\P_{n,m}\left[\max_{1\le j\le m}|z_j|>1+\eta_\alpha\right]\le\frac{e}{2\pi}\sqrt{\frac{m}{1-\alpha}}e^{-m}.\]
In particular, we may take
$1+\eta_\alpha=\sqrt{3+\log(\alpha^{-1})}$, (although when $\alpha\to 1$,
$\eta_\alpha$ may in fact be taken to be any fixed positive number).  Take $\epsilon\in\left(0,\left(\frac{\alpha}{4+2\sqrt{\alpha}}\right)^2\right)$ as in Lemma \ref{L:reg}, such that also $1+\eta_\alpha<\frac{1}{\sqrt{\alpha}}-\sqrt{\epsilon}$. Let
\[A_{\alpha,r}:=\left\{(z_1,\ldots,z_m)\in\C^m:|z_j|<1+\eta_\alpha, 
  j=1,\ldots,n,\mbox{ and } d_{BL} \left( \hat{\mu}_{\bf z}, \mu_{\alpha} \right) \geq r \right\}.\]
The probability that the eigenvalues of the top-left $m\times m$ block
of $\sqrt{\frac{n}{m}}U$ lie in the set $A_{\alpha,r}$ is given by
\begin{equation}\begin{split}\label{E:prob-A_r}
\mathbb{P}_{n,m} (A_{\alpha,r}) &= \frac{1}{Z_{n,m}} \int_{A_{\alpha,r}} e^{-H_{n,m}( z_1, \ldots, z_m) } \prod_{i=1}^m d\lambda(z_i) \\
&\leq \frac{1}{Z_{n,m}} \int_{A_{\alpha,r}} e^{-\left[ m^2 \mathcal{E}_{V_{n,m}} \left( \hat{\mu}_{\bf z}^{(\epsilon)} \right) - m \mathcal{E} \left( \lambda_{\epsilon} \right) - \frac{m^2(1-\alpha)\epsilon}{2\alpha}\right] } \prod_{i=1}^m d\lambda(z_i) \\
&\leq \frac{1}{Z_{n,m}} e^{-\left[ m^2 \inf_{A_{\alpha,r}} \mathcal{E}_{V_{n,m}} \left( \hat{\mu}_{\bf z}^{(\epsilon)} \right) - m \mathcal{E} \left( \lambda_{\epsilon} \right) - \frac{m^2(1-\alpha)\epsilon}{2\alpha}\right] } \left(\frac{\pi}{\alpha}\right)^m,
\end{split}\end{equation}
by Lemma \ref{L:reg}.

The normalizing constant $Z_{n,m}$ can be bounded in terms of $\mu_\alpha$ as follows:
\begin{align*}
\log&(Z_{n,m})\\&=\log\idotsint e^{-\sum_{j\neq k}\log\left(\frac{1}{|z_j-z_k|}\right)+(n-m-1)\sum_{j=1}^m\log\left(1-\alpha|z_j|^2\right)}d\lambda(z_1)\cdots d\lambda(z_m)\\&=\log\idotsint e^{-\sum_{j\neq k}\log\left(\frac{1}{|z_j-z_k|}\right)+(n-m+1)\sum_{j=1}^m\log\left(1-\alpha|z_j|^2\right)}\left(\frac{1}{2(1-\alpha)}\right)^m d\mu_\alpha(z_1)\cdots d\mu_\alpha(z_m)\\&\ge -m\log(2(1-\alpha)) \\&\qquad+\int \left[-\sum_{j\neq k}\log\left(\frac{1}{|z_j-z_k|}\right)+(n-m+1)\sum_{j=1}^m\log\left(1-\alpha|z_j|^2\right)\right] d\mu_\alpha(z_1)\cdots d\mu_\alpha(z_m)\\&=-m\log(2(1-\alpha))-m(m-1)\mathcal{E}(\mu_\alpha)+\left(\frac{n-m+1}{n-m-1}\right)m^2\int V_{n,m}d\mu_\alpha,
\end{align*}
where the inequality is by Jensen's inequality.

Observe that
\begin{equation*}\begin{split}
m^2 \inf_{A_{\alpha,r}} \mathcal{E}_{V_{n,m}} \left( \hat{\mu}_{\bf
    z}^{(\epsilon)} \right) &= m^2 \inf_{A_{\alpha,r}} \left(
  \mathcal{E}_{V_{n,m}} \left( \hat{\mu}_{\bf z}^{(\epsilon)} \right)
  -\mathcal{E}_{V_{n,m}} \left(\mu_{\alpha} \right)\right) +m^2
\mathcal{E}_{V_{n,m}} \left(\mu_{\alpha} \right)\\ &= m^2
\inf_{A_{\alpha,r}} \left( \mathcal{E}_{V_{n,m}} \left( \hat{\mu}_{\bf
      z}^{(\epsilon)} \right) -\mathcal{E}_{V_{n,m}}
  \left(\mu_{\alpha} \right)\right)
+m^2\left[\mathcal{E}(\mu_\alpha)+\int V_{n,m}d\mu_\alpha\right],
\end{split}\end{equation*}
and by Lemma \ref{L:transport}, 
\begin{align*}
\inf_{A_{\alpha,r}} \left( \mathcal{E}_{V_{n,m}} \left( \hat{\mu}_{\bf z}^{(\epsilon)} \right) - \mathcal{E}_{V_{n,m}} \left( \mu_{\alpha} \right) \right) &\geq \frac{1}{16 \pi (1+\eta_\alpha+\epsilon)^2} \inf_{A_{\alpha,r}} d_{BL} \left( \hat{\mu}_{\bf z}^{(\epsilon)}, \mu_{\alpha} \right)^2.
\end{align*}
Since $d_{BL}(\hat{\mu}^{(\epsilon)}_{\bf z},\hat{\mu}_{\bf z})\le\epsilon$, 
\begin{equation*}\begin{split}
\frac{1}{2}d_{BL} \left( \hat{\mu}_{\bf z}, \mu_{\alpha} \right)^2 & \leq \frac{1}{2} \left( d_{BL} \left(\hat{\mu}_{\bf z}^{(\epsilon)}, \mu_{\alpha} \right) + d_{BL} \left( \hat{\mu}_{\bf z}^{(\epsilon)}, \hat{\mu}_{\bf z} \right) \right)^2 \\
& \leq d_{BL} \left( \hat{\mu}_{\bf z}^{(\epsilon)}, \mu_{\alpha} \right)^2 + d_{BL} \left( \hat{\mu}_{\bf z}^{(\epsilon)}, \hat{\mu}_{\bf z} \right)^2 \leq d_{BL} \left( \hat{\mu}_{\bf z}^{(\epsilon)}, \mu_{\alpha} \right)^2 + \epsilon^2.
\end{split}\end{equation*}
It follows that 
\begin{equation*}
\inf_{A_{\alpha,r}} \left( \mathcal{E}_{V_{n,m}} \left( \hat{\mu}_{\bf z}^{(\epsilon)} \right) - \mathcal{E}_{V_{n,m}} \left( \mu_{\alpha} \right) \right) \geq \frac{1}{16\pi(1+\eta+\epsilon)^2} \left( \frac{r^2}{2} - \epsilon^2 \right),
\end{equation*}
and so combining the estimate in \eqref{E:prob-A_r} with the analysis above gives that 
\begin{align*}\P_{n,m}(A_{\alpha,r})&\le
\exp\left\{-\frac{
                                      m^2}{16\pi(1+\eta_\alpha+\epsilon)^2}\left(\frac{r^2}{2}-\epsilon^2\right)-m\mathcal{E}(\mu_\alpha)-\frac{2(n-m)m^2}{n-m-1}\int V_{n,m}d\mu_\alpha\right.\\&\qquad \qquad \qquad\left.+m\left(\log\left(\frac{2\pi(1-\alpha)}{\alpha }\right)\right) +\frac{m^2(1-\alpha)\epsilon}{2\alpha}+m\mathcal{E}(\lambda_\epsilon)\right\}\\&\le \exp\left\{-\frac{ m^2}{16\pi(1+\eta_\alpha+\epsilon)^2}\left(\frac{r^2}{2}-\epsilon^2\right)+m\left(\log\left(\frac{2\pi(1-\alpha)}{\alpha }\right)\right) \right.\\&\qquad\qquad \qquad\qquad \qquad\qquad\qquad\qquad \qquad\qquad\left.+\frac{m^2(1-\alpha)\epsilon}{2\alpha}+m\mathcal{E}(\lambda_\epsilon)\right\}.\end{align*}

Now, the Coulomb energy of $\lambda_{\epsilon}$ is 
\begin{equation*}\begin{split}
\mathcal{E}(\lambda_{\epsilon}) 
& = \int \int \log \left( \frac{1}{|\epsilon x - \epsilon y|} \right) d\lambda_1(x) d\lambda_1 (y)=-\log(\epsilon) + \mathcal{E}(\lambda_1) = -\log(\epsilon)+ \frac{1}{4}=\log\left(\frac{e^{1/4}}{\epsilon}\right), 
\end{split}\end{equation*}
and so 

\begin{equation}\begin{split}\label{E:final-bound}\P_{n,m}&(A_{\alpha,r})\\&\le \exp\left\{-\frac{m^2}{16\pi(1+\eta_\alpha+\epsilon)^2}\left(\frac{r^2}{2}-\epsilon^2\right)+m\left(\log\left(\frac{2\pi
                                    e^{1/4}(1-\alpha)}{\alpha\epsilon }\right)\right)+\frac{m^2(1-\alpha)\epsilon}{2\alpha} \right\}.\end{split}\end{equation}

Now take
\[\epsilon=\min\left\{\left(\frac{\alpha}{2(2+\sqrt{\alpha})}\right)^2,\left(\frac{1}{\sqrt{\alpha}}-(1+\eta_\alpha)\right)^2,\frac{8\sqrt{\pi}(1+\eta_\alpha)}{m}, 1+\eta_\alpha,\frac{2\alpha}{(1-\alpha)m^2}\right\}.\]
The analysis above required that
$\epsilon\le\left(\frac{\alpha}{4+2\sqrt{\alpha}}\right)^2$ and that
$1+\eta_\alpha\le\frac{1}{\sqrt{\alpha}}-\sqrt{\epsilon}$, which is
guaranteed by the first two terms in the minimum.  

For the first term in the estimate \eqref{E:final-bound}, first
estimating the $\epsilon$ in the denominator by $1+\eta_\alpha$ and
then the factor of $\epsilon^2$ in the numerator by
$\frac{8\sqrt{\pi}(1+\eta_\alpha)}{m}$ gives that
\[-\frac{m^2}{16\pi(1+\eta_\alpha+\epsilon)^2}\left(\frac{r^2}{2}-\epsilon^2\right)\le
  -\frac{m^2}{64\pi(1+\eta_\alpha)^2}\left(\frac{r^2}{2}-\epsilon^2\right)\le -\frac{m^2r^2}{128\pi(1+\eta_\alpha)^2}+1.\]
For the second term of \eqref{E:final-bound}, one verifies each of the
5 possible choices of $\epsilon$ above: in all cases, 
\[m\left(\log\left(\frac{2\pi
                                    e^{1/4}(1-\alpha)}{\alpha\epsilon
                                  }\right)\right)\le
                              2m\log(m)+mC_\alpha',\]
where 
\begin{align*}C_{\alpha}'&=\max\left\{2\log(6)+3\log(\alpha^{-1}),
  -2\log(1-\sqrt{\alpha}(1+\eta_\alpha)),
  -\log(\alpha(1+\eta_\alpha))\right\}+\log(2\pi e^{1/4})\\&\le2\log(6)+3\log(\alpha^{-1})
                                                            +\log(2\pi
                                                             e^{1/4})\\&\le
                                                            6+3\log(\alpha^{-1}),\end{align*}
since $\eta_\alpha>0$ and $\sqrt{\alpha}\eta_\alpha\le 1$.
Since we take $\epsilon\le\frac{2\alpha}{(1-\alpha)m^2}$, the final
term inside the exponential in the estimate \eqref{E:final-bound} is
bounded by 1.  All together, then
\[\P_{n,m}(A_{\alpha,r})\le \exp\left\{-C_\alpha m^2r^2+2m\log(m)+mC_\alpha'+2\right\},\]
where
$C_\alpha=\frac{1}{128\pi(1+\eta_\alpha)^2}\ge\frac{1}{128\pi(1+\sqrt{3+\log(\alpha^{-1})})^2}$
and $C_\alpha'=6+3\log(\alpha^{-1})$.
\end{proof}

The proof of Theorem \ref{T:union} then follows from Theorem \ref{T:main} and an application of the Borel-Cantelli Lemma.
\begin{proof}[Proof of Theorem \ref{T:union}]
If $m \geq \frac{n}{e}$, then in Theorem \ref{T:main}, $C_\alpha$ can
be taken to be $\frac{1}{1152\pi}$ and $C_\alpha'$ can be taken to be
$9$.  Choosing
\begin{align*}
\delta_m & =  \sqrt{ \frac{ 4\log m}{C_\alpha m}}=48\sqrt{\frac{2\pi \log(m)}{m}},
\end{align*}
Theorem \ref{T:main} gives that
\begin{align*}
\mathbb{P} \left[ d_{BL} \left( \hat{\mu}_m , \mu_{\alpha} \right) \geq \delta_m \right] & \leq e^{\left\{2- C_{\alpha}m^2 \delta_m^2 +2 m\log m + C_{\alpha}'m \right\} } +\frac{e}{2\pi}\sqrt{\frac{m}{1-\alpha}}e^{-m}\\
& \le e^{2-2m\log(m)+C_\alpha'm}+ne^{-\frac{n}{e}},
\end{align*}
which is summable.

If instead $m < \frac{n}{e}$,  then in Theorem \ref{T:main} we may
take $C_\alpha=\frac{1}{128\pi(1+\sqrt{3\log(\alpha^-1)})^2}$ and
$C_\alpha'=9\log(\alpha^{-1})$.  Choosing
\begin{align*}
\delta_m & = \sqrt{ \frac{ 9\log n}{C_\alpha m}} \leq \frac{24(1+\sqrt{3\log(\frac{n}{m})})\sqrt{2\pi\log(n)}}{\sqrt{m}}\le\frac{165\sqrt{\log(\frac{n}{m})\log(n)}}{\sqrt{m}},
\end{align*}
Theorem \ref{T:main} gives that
\begin{align*}
\mathbb{P} \left[ d_{BL} \left( \hat{\mu}_m , \mu_{\alpha} \right) \geq \delta_m \right] & \leq e^{\left\{2- C_{\alpha}m^2 \delta_m^2 + 2m\log m + mC_{\alpha}' \right\} }+\frac{e}{2\pi}\sqrt{\frac{m}{1-\alpha}}e^{-m}
\\ & \le e^{2-7m\log(m)} +\frac{e}{2\pi}\sqrt{\frac{m}{1-\alpha}}e^{-m},
\end{align*}
which is summable since $m\ge k_n$ and $\frac{k_n}{\log(n)^2}\to\infty$.
The claimed result thus follows from the Borel-Cantelli lemma.
\end{proof}
\bibliographystyle{plain}
\bibliography{truncations}

\end{document}